\documentclass[12pt]{amsart}

\usepackage{psfrag}
\usepackage{color}
\usepackage{tikz}
\usetikzlibrary{matrix,arrows}
\usepackage{graphicx,graphics}
\usepackage{fullpage,amssymb,amsfonts,amsmath,amstext,amsthm,amscd,verbatim,enumerate}
\usepackage[T1]{fontenc}

\begin{document}

\newtheorem{theorem}{Theorem}[section]
\newtheorem{result}[theorem]{Result}
\newtheorem{fact}[theorem]{Fact}
\newtheorem{conjecture}[theorem]{Conjecture}
\newtheorem{lemma}[theorem]{Lemma}
\newtheorem{proposition}[theorem]{Proposition}
\newtheorem{corollary}[theorem]{Corollary}
\newtheorem{facts}[theorem]{Facts}
\newtheorem{props}[theorem]{Properties}
\newtheorem*{thmA}{Theorem A}
\newtheorem{ex}[theorem]{Example}
\theoremstyle{definition}
\newtheorem{definition}[theorem]{Definition}
\newtheorem{remark}[theorem]{Remark}
\newtheorem{example}[theorem]{Example}
\newtheorem*{defna}{Definition}

\newcommand{\notes} {\noindent \textbf{Notes.  }}
\newcommand{\note} {\noindent \textbf{Note.  }}
\newcommand{\defn} {\noindent \textbf{Definition.  }}
\newcommand{\defns} {\noindent \textbf{Definitions.  }}
\newcommand{\x}{{\bf x}}
\newcommand{\z}{{\bf z}}
\newcommand{\B}{{\bf b}}
\newcommand{\V}{{\bf v}}
\newcommand{\T}{\mathbb{T}}
\newcommand{\Z}{\mathbb{Z}}
\newcommand{\Hp}{\mathbb{H}}
\newcommand{\D}{\mathbb{D}}
\newcommand{\R}{\mathbb{R}}
\newcommand{\N}{\mathbb{N}}
\renewcommand{\B}{\mathbb{B}}
\newcommand{\C}{\mathbb{C}}
\newcommand{\ft}{\widetilde{f}}
\newcommand{\dt}{{\mathrm{det }\;}}
 \newcommand{\adj}{{\mathrm{adj}\;}}
 \newcommand{\0}{{\bf O}}
 \newcommand{\av}{\arrowvert}
 \newcommand{\zbar}{\overline{z}}
 \newcommand{\xbar}{\overline{X}}
 \newcommand{\htt}{\widetilde{h}}
\newcommand{\ty}{\mathcal{T}}
\renewcommand\Re{\operatorname{Re}}
\renewcommand\Im{\operatorname{Im}}
\newcommand{\tr}{\operatorname{Tr}}

\newcommand{\ds}{\displaystyle}
\numberwithin{equation}{section}

\renewcommand{\theenumi}{(\roman{enumi})}
\renewcommand{\labelenumi}{\theenumi}

\title{On the infinitesimal space of UQR mappings}

\author{Alastair Fletcher}
\address{Department of Mathematical Sciences, Northern Illinois University, DeKalb, IL 60115-2888. USA}
\email{fletcher@math.niu.edu}
\author{Doug Macclure}
\email{dmacclure1@niu.edu}
\author{James Waterman}
\email{watermanjamesa@gmail.com}
\author{Sarah Wesley}
\email{swesley@niu.edu}

\thanks{This work was supported by a grant from the Simons Foundation (\#352034, Alastair Fletcher). This work arose from a project given by the first author at a graduate level course at Northern Illinois University.}

\maketitle

\begin{abstract}
Generalized derivatives and infinitesimal spaces generalize the idea of derivatives to mappings which need not be differentiable. It is particularly powerful in the context of quasiregular mappings, where normal family arguments imply generalized derivatives always exist. The main result of this paper is to show that if $f$ is any uniformly quasiregular mapping with $x_0$ a topologically attracting or repelling fixed point, at which $f$ is locally injective, then $f$ may be conjugated to a uniformly quasiregular mapping $g$ with fixed point $0$ and so that the infinitesimal space of $g$ at $0$ contains uncountably many elements. This should be contrasted with the fact that $f$ (and also $g$) is conjugate to $x\mapsto x/2$ or $x\mapsto 2x$ in the attracting or repelling cases respectively.
\end{abstract}

\section{Introduction}

\subsection{Motivation}

The notion of the derivative has driven the development of modern mathematics since Newton and Leibniz first developed it.  It has led to countless applications in various fields of mathematics and physics. Unfortunately, not all functions are differentiable.  In fact, one can argue that most functions are not differentiable, let alone differentiable everywhere on the function's domain.  Points where a function is not differentiable cannot always be ignored, particularly when one is concerned with the dynamics of a function near a fixed point.  With this, we are given impetus to further study the theory of infinitesimal spaces, constructed for quasiregular mappings by Gutlyanski, Martio, Ryazunov, and Vuorinen in \cite{GMRV}.  

Consider a function $f: \R \to \R$. The definition of $f$ being differentiable at $x_0\in \R$ implies that $f$ is well-approximated near $x_0$ by the map 
\[x \mapsto f(x_0) + f'(x_0)(x-x_0)\]
on a neighborhood of $x_0$.  
Indeed, for a differentiable function $f: \R^m \to \R^n$, differentiability at $x_0 \in \R^m$ implies that $f$ is well-approximated by the map 
\[x \mapsto f(x_0) + f'(x_0)(x-x_0)\]
on a neighborhood of $x_0$, where $f'(x_0)$ is the derivative matrix of $f$.  

Viewing the derivative of a holomorphic map $f:\C \to \C$ at $z_0$ as a linear map $\R^2 \to \R^2$, we obtain a more specific form than an arbitrary linear map.  Such derivatives are $2\times 2$ matrices of the form
\[
\left (
\begin{array}{ll}
a & b \\
-b & a 
\end{array}
\right )
= 
\left (
\begin{array}{ll}
|f'(z_0)| & 0 \\
0 & |f'(z_0)| 
\end{array}
\right ) \circ
\left (
\begin{array}{ll}
\cos \arg f'(z_0) & -\sin \arg f'(z_0) \\
\sin \arg f'(z_0) & \cos \arg f'(z_0)
\end{array}
\right )
\]
since the Cauchy-Riemann equations are satisfied.  Thus, the resulting approximating linear map can be written as a composition of a scaling and a rotation, which maps infinitesimal circles to infinitesimal circles.  

Now, for complex-valued functions, if we allow arbitrary non-singular linear maps as derivatives, the derivative maps circles to ellipses of uniformly bounded eccentricity.  Such maps are basic examples of quasiconformal maps.  See for example \cite{Alf,FM} for a more detailed development of the theory of quasiconformal maps.  

\subsection{Generalized Derivatives}

Now, we wish to show how we can generalize differentiability for maps that aren't differentiable.  Gutlyanski et al \cite{GMRV} define the notion of the generalized derivative at $x_0 \in \R^d$ as locally uniformly convergent limits of the family 
\[ \mathcal{F}_{x_0} =\left \{F_\delta(x) = \frac{f(\delta x + x_0) - f(x_0)}{\rho_f(\delta)}: \delta > 0\right \},\]
where 
\[\rho_f(\delta) = \left ( \frac{\mu [f(B(x_0,\delta))]}{\mu [B(0,1)]} \right )^{1/d}\]
is the mean radius of the unit ball under $f$ and $\mu$ is the standard Lebesgue measure.  The infinitesimal space of a map $f$ at $x_0$, denoted by $T(x_0, f)$, consists of limits of subsequences from $\mathcal{F}_{x_0}$. For quasiregular mappings (see below for the definition), normal family arguments imply $T(x_0,f)$ is always non-empty.

\begin{example}
To illustrate generalized derivatives, consider the function $f: \R \to \R$ defined by
\[ f(x) = 
\left \{ 
\begin{array}{ll}
x, & x \geq 0 \\
x/2, & x < 0.
\end{array}
\right .
\]
If we zoom in at the origin, the geometry of the graph remains the same.  So, no matter how far we zoom in, $f$ will always appear to be a piecewise linear function.  Now, pick $\delta > 0$.  Then the map $f(\delta x)$ maps the unit ball of $\R$, that is $(-1,1)$, onto the open interval $(-\delta/2, \delta)$.  We then have
\[\rho(\delta) = \frac{\delta + \delta/2}{2} = \frac{3\delta}{4}\]
and 
\[ g_\delta(x) = \frac{f(\delta x)}{\rho(\delta)} = 
\left \{ 
\begin{array}{ll}
4x/3, & x \geq 0 \\
2x/3, & x < 0,
\end{array}
\right .\]
which tells us the "shape" of $f$ on a scale of $\delta$, but not how much $f$ shrinks or enlarges.
Now, consider limits $g_{\delta_k}(x)$ for a sequence $\delta_k > 0$, $\delta_k \to 0$.  
Since $g_{\delta}$ is independent of $\delta$, there exists only one limit map, $g = g_\delta$. Hence the infinitesimal space contains only one mapping.  
\end{example}

\begin{definition}
If $T(x_0,f)$ consists of only one mapping, then $f$ is called simple at $x_0$.
\end{definition}

We remark that this terminology is slightly different to that used in \cite{GMRV}, where the infinitesimal space itself was called simple. Roughly speaking, if $f$ is simple at $x_0$, then $f$ is well-behaved near $x_0$. We first recall some notation before expanding on this idea.

Throughout, we will be using the equivalence relation $\sim$ as in \cite{GMRV}. If $U \subset \R^d$ is a neighborhood of $0$ and
$v,w: U \to \R^d$ are mappings, then
\[ v(x) \sim u(x) \]
as $x\to 0$ if
\[ ||v(x) - u(x)|| = o(||v(x)|| + ||u(x)||),\]
where $\alpha = o(\beta)$ means that given $\epsilon >0$, there is a neighborhood $V$ of $0$ such that $|\alpha(x)| \leq \epsilon |\beta(x)|$ for $x\in V$.  Here, we are using the standard Euclidean norm $|\cdot|$. 

\begin{remark}
With the notation as above,
\begin{enumerate}[(i)]
\item $g_\delta$ always preserves the measure of the unit ball, so any $\psi \in T(x_0,f)$ does too.  
\item If $f$ is differentiable at $x_0$, $T(x_0,f)$ consists only of a scaled version of the derivative, with the scaling so that the measure of the unit ball is preserved. In this case, $f$ is simple at $x_0$.
\item  If $f$ is simple at $x_0$ with $T(x_0,f) = \{ \psi \}$, then 
\begin{enumerate}
\item by \cite[Proposition 4.7]{GMRV}, \lefteqn{ f(x) \sim \rho(|x|)\psi(x/|x|),}
\item by \cite[Theorem 4.1]{GMRV}, for all $t > 0$, $x \in \R^d\setminus \{0\}$, there exists $D > 0$ such that $\psi(tx) = t^D\psi(x)$.  In other words, $\psi$ is $D$-positive homogeneous.
\end{enumerate}
\end{enumerate}
\end{remark}

Note, the latter part of the remark implies that near a simple point $x_0$ a quasiregular map maps any ball of sufficiently small radius centered at $x_0$ to a topological surface which can be radially shrunk or expanded to the topological surface $\psi(S^{d-1})$.  This is because asymptotically, $f$ can be expressed as a radial contraction of $\psi(S^{d-1})$.  

\subsection{Generalized Derivatives of Quasiregular Mappings}

Next, we wish to apply the concept of generalized derivatives to quasiconformal and quasiregular maps (see \cite{Rickman} for a detailed monograph on quasiregular mappings). Quasiregular mappings are Sobolev mappings in $W^1_{d,loc}(\mathbb{R}^d)$ where there is a uniform bound on distortion. More precisely, a mapping $f: E \rightarrow {\R}^d$ defined on a domain $E \subset {\R}^d$ is quasiregular if $f$ belongs to the above Sobolev space and there exists $K \in [1,\infty)$ such that 
\begin{equation} 
\label{eq:1}
|f'(x)|^d \leq KJ_f(x) 
\end{equation}
almost everywhere.  Here, $J_f(x)$ denotes the Jacobian determinant of $f$ at $x \in E$.  The smallest constant $K \geq 1$ for which \eqref{eq:1} holds is called the outer distortion $K_O(f)$.  If $f$ is quasiregular, then we also have 
\begin{equation}
\label{eq:2}
J_f(x) \leq K' \inf_{|h| = 1}|f'(x)h|^d
\end{equation}
almost everywhere in $E$ for some $K' \in [1,\infty).$  The smallest constant $K' \geq 1$ for which \eqref{eq:2} holds is called the inner distortion $K_I(f).$  The maximal distortion $K = K(f)$ of $f$ is the larger of $K_O(f)$ and $K_I(f)$, and we then say that $f$ is $K$-quasiregular. Quasiconformal mappings are injective quasiregular mappings.

Given a quasiregular map $f : E \to \R^d$ and $x \in E$, for $r$ such that $0 < r < d(x,\partial E)$ define
\[L(r) = \max_{|x-x_0| = r}\{|f(x) - f(x_0)|\}, \]
\[ l(r) = \min_{|x-x_0| = r}\{|f(x) - f(x_0)|\}\]
and
\[ H(x_0) = \limsup_{r\to 0}\frac{L(r)}{l(r)}.\]
We call $H(x_0)$ the {\it linear distortion} of $f$ at $x_0$. By \cite[Theorem II.4.3]{Rickman}, 
\[ H(x_0) \leq C = C(i(x_0,f)K_O(f)),\] 
where $i(x_0,f)$ denotes the local (topological) index of $f$ at $x_0$. 

It is easy to see that $l(r) \leq \rho(r) \leq L(r)$.  Since $H(x_0)$ is bounded above and, by \cite[Theorem III.4.7]{Rickman}, quasiregular maps satisfy a bi-H\"older condition, $\rho$ cannot behave too poorly.  

The following lemma is part of \cite[Theorem 2.7]{GMRV}.

\begin{lemma}
Let $f: \R^d \to \R^d$ be $K$-quasiregular.  Then for any $x_0\in \R^d$, $T(x_0,f)$ is nonempty. 
\end{lemma}

Our main result concerns a special sub-class of quasiregular mappings.

\begin{definition} A quasiregular map $f: \R^d \to \R^d$ is called {\it uniformly $K$-quasiregular} if the maximal distortion of the iterates of $f$ satisfy $K(f^n) \leq K $ for all $n \geq 1.$
\end{definition}

We will abbreviate the term uniformly $K$-quasiregular to $K$-uqr, or simply uqr if we do not need to specify the bound on the distortion of the iterates.

In \cite{HMM}, Hinkkanen et al considered generalized derivatives arising as limits of 
$\frac{f(\lambda_k x)}{\lambda_k}$
as $\lambda_k \to 0$.  They used the fact that uqr maps are bi-Lipschitz on a neighborhood of $x_0$ if $i(x_0,f) = 1$. In general quasiregular mappings are only H\"older continuous and so the idea of infinitesimal spaces as defined above is the appropriate generalization.

Further, in \cite{HMM} a classification for fixed points of uqr mappings is given in analogy with that for holomorphic functions.
If we assume that $i(x_0,f)=1$, then $x_0$ is called attracting or repelling respectively if the infinitesimal space as defined in \cite{HMM} (which we remind the reader is different from that considered in this paper and, in particular, the definition given here does not distinguish between attracting and repelling fixed point) consists only of loxodromically  attracting or loxodromically repelling uniformly quasiconformal mappings respectively.
A map $\psi : \R^d \to \R^d$ is called loxodromically attracting or loxodromically repelling if $\psi$ fixes $0$ and $\infty$, and $\psi^m(x) \to 0$ for $x \in \R^n$ or $\psi^m(x) \to \infty$ for $x\neq 0$ respectively.  

We will need the following lemma on radial mappings, which we record here for the sequel. This is well-known, see for example \cite[Example 6.5.1]{IM}, but we include a proof for the convenience of the reader. 

\begin{lemma}
\label{lem:radial}
Let $B^d$ denote the unit ball in $\R^d$ and let $\alpha > 0$. The radial map given by $F(x) = x|x|^{\alpha-1}$ is $K$-quasiconformal with $K(F) = \max\{ \alpha^{d-1}, \alpha^{1-d} \} .$
\end{lemma}

\begin{proof}
Let $x\in B^d \setminus \{ 0 \}$. By radial symmetry, it suffices to assume that $x = (t,0,\ldots,0)$ for $t\in (0,1)$. Then $F(x) = (t^{\alpha},0,\ldots,0)$ and
\[F'(x)=\left[
 \begin{array}{ccccc}
   \alpha t^{\alpha -1}&&\text{\huge0}\\
    & t^{\alpha -1} & \\
     \text{\huge0}& & \ddots\\
 \end{array}.
\right]\]
Therefore the Jacobian is $J_F(x) = \alpha t^{d(\alpha-1)}$. We first assume $\alpha \geq 1$. Then the norm of the derivative is $|F'(x)| = \alpha t^{\alpha -1}$ and so
\[ \frac{|F'(x)|^d}{J_F(x)} = \alpha^{d-1}.\]
Similarly, $\ell(F'(x)) = \inf_{|h| = 1}|F'(x)h|^d = t^{\alpha -1}$ and so
\[ \frac{ J_F(x)}{\ell(F'(x))^d} = \alpha.\]
Since $0$ is removable for quasiconformal mappings, we have $K_O(F) = \alpha ^{d-1}$ and $K_I(F) = \alpha$, recalling \eqref{eq:1} and \eqref{eq:2}. In particular, $F$ is $\alpha^{d-1}$-quasiconformal.

On the other hand, if $0<\alpha <1$, then a similar calculation shows that $K_O(F) = \alpha^{-1}$ and $K_I(F) = \alpha^{1-d}$ and so $F$ is $\alpha^{1-d}$-quasiconformal.
\end{proof}

\subsection{Statement of results}

This paper will deal with the question of whether or not mappings are always simple everywhere. As we have seen above, holomorphic mappings are simple everywhere, but quasiregular mappings are only guaranteed to be simple almost everywhere.

\begin{proposition}
\label{prop:1}
Let $d\geq 2$, $K> 1$ and denote by $B^d$ the unit ball in $\R^d$.
Then there exists a $K$-quasiconformal map $F:B^d \to B^d$ so that $F$ is not simple at $0$.
\end{proposition}

This result is not hard to prove and is undoubtedly known to experts in the field. We will provide a radial example to motivate what will follow in the sequel.

Our main result concerns uniformly quasiregular mappings. By a result of Hinkkanen and Martin \cite{HM}, if $x_0$ is a repelling fixed point of a uqr map $f$, then there is a quasiconformal map $L$ which conjugates $f$ in a neighborhood of $x_0$ to the map $x\mapsto 2x$. Correspondingly, by taking inverses, if $x_0$ is an attracting fixed point of a uqr map $f$, then $f$ can be locally quasiconformally conjugated to $x \mapsto x/2$. This latter map is clearly simple, and so perhaps this forces $f$ to be simple at $x_0$. We will show this is not the case.

\begin{theorem}
\label{thm:main}
There exists a uniformly quasiconformal map $H:B^d \to B^d$ with an attracting fixed point at $x_0=0$ and so that $f$ is not simple at $0$.
\end{theorem}

Our construction is based on conjugating $x\mapsto x/2$ by the quasiconformal map we will construct in Proposition \ref{prop:1}. In the proof of Theorem \ref{thm:main}, we will exhibit two distinct elements of $T(H,0)$, but in fact more is true.

\begin{corollary}
\label{cor:0}
Both $F$ and $H$, defined as in Proposition \ref{prop:1} and Theorem \ref{thm:main} respectively, contain uncountably many elements in their respective infinitesimal spaces at $0$.
\end{corollary}

It would be interesting to know if $T(f,x_0)$ is always either simple or contains uncountably many elements. This will be the subject of future work.

Finally, as a corollary to Theorem \ref{thm:main}, we will show how any uniformly quasiregular map with an attracting or repelling fixed point can be conjugated to a non-simple one.

\begin{corollary}
\label{cor:1}
Let $f$ be a uniformly quasiregular mapping, $x_0$ a fixed point of $f$ with $i(x_0,f) = 1$ and suppose $f$ is either repelling or attracting at $x_0$. Then there exists a quasiconformal map $g$ so that $g \circ f \circ g^{-1}$ is not simple at $g(x_0)$.
\end{corollary}

\section{Infinitesimal spaces and quasiconformal maps}

In this section, we give an example to show that it is possible for a quasiconformal mapping to not be simple at a given point.

\begin{proof}[Proof of Proposition \ref{prop:1}]
Fix $d\geq 2$, $K>1$ and let $B^d$ be the unit ball in $\R^d$. We will construct a radial quasiconformal map from $B^d$ to itself which is not simple at $0$. To that end, we will construct a continuous increasing function $f:[0,1] \to [0,1]$ for which the infinitesimal space is not simple, and then radially extend to a quasiconformal mapping.

The main idea is to construct a decreasing sequence $r_n \to 0$ so that $f(r_n) = 2^{-n}$ and the behavior of $f$ is specified by exponents $k_n$ and constants $C_n$ so that $f(r) = C_nr^{k_n}$ on subintervals of the form $[r_n,r_{n-1}]$. We fix $f(1)=1$ so that $r_0=1$ and $C_1=1$. 

Next, by continuity,
\begin{align*}
C_2 r_{1}^{k_2} &=C_1 r_{1}^{k_1} = r_1^{k_1} \Longrightarrow C_2 = r_1^{k_1 - k_2} \\
C_3 r_{2}^{k_3} &=C_2 r_{2}^{k_2} \Longrightarrow C_3 = r_1^{k_1 - k_2} r_2^{k_2 - k_3} \\
C_n &= r_1^{k_1 - k_2} r_2^{k_2 - k_3} ... r_{n-1}^{k_{n-1} - k_n} = \prod_{m=1}^{n-1} r_m^{k_{m} -k_{m+1}}
\end{align*}
Further, note that $\frac{C_n}{C_{n+1}} = \frac{1}{r_{n}^{k_n - k_{n+1}}} = r_{n}^{k_{n+1}-k_n}$.
Since any $f$, as above, is an increasing function of $r$, we may choose $r_n$ such that $f(r_n)= \frac{1}{2^n}$. Then
\begin{equation}
\label{eq:Cr}
C_n = \left ( \frac{1}{2} \right )^n r_n^{-k_n}.
\end{equation}
Hence
\[f(r_n)=\left(C_{n+1} \right) r_n^{k_{n+1}} =\left(\prod_{m=1}^{n} r_m^{k_{m} -k_{m+1}} \right) r_n^{k_{n+1}} = \frac{1}{2^n}\]
and
\[f(r_{n+1})=\left(\prod_{m=1}^{n+1} r_m^{k_{m} -k_{m+1}}\right) r_{n+1}^{k_{n+2}} =\left(\prod_{m=1}^{n} r_m^{k_{m} -k_{m+1}}\right) r_{n+1}^{k_{n+1}} .\]
Thus,
\[f(r_{n+1})=f(r_n)\left(\frac{r_{n+1}}{r_n}\right)^{k_{n+1}} \]
and therefore
\begin{equation} \label{rn}
f(r_{n+1})=\left(\frac{1}{2}\right)^{n+1}=\left(\frac{1}{2}\right)^{n}\left(\frac{r_{n+1}}{r_n}\right)^{k_{n+1}}\Longrightarrow r_{n}= \left(\frac{1}{2}\right)^{\frac{1}{k_{n}}} r_{n-1}
 \end{equation}
and
\[r_n = \left(\frac{1}{2}\right)^{\frac{1}{k_{n}} +\frac{1}{k_{n-1}} + ... + \frac{1}{k_1} }.\]
If $(k_n)_{n=1}^{\infty}$ is a bounded sequence of positive real numbers, then $r_n \downarrow 0$. In conclusion $f:[0,1]\to [0,1]$ is a homeomorphism.

While the general construction above is more flexible, we will use the explicit sequence $(k_n)_{n=1}^{\infty}$ given by 
\begin{equation}
\label{k}
k_{2n-1}=K \text{ and } k_{2n}=\frac{1}{K}
\end{equation} 
for a fixed $K>1$. With this choice, we have
\begin{equation}\label{eq:r}
r_{2n} = 2^{-(nK + n/K)}, \quad, r_{2n - 1} = 2^{-((n - 1)K + n/K)},
\end{equation}
\begin{equation} \label{eq:C} 
C_{2n} = 2^{n(1/K^2 - 1)}, \quad C_{2n - 1}= 2^{(n - 1)(K^2 - 1)}.
\end{equation}
We will need to have formulas for products in the sequence $(r_n)_{n=1}^{\infty}$. In particular, one can check that
\begin{equation}
\label{eq:rs}
r_{2n} r_{m}=  r_{2n+m}
\end{equation}
for any $n,m\in \N$. On the other hand, 
\begin{equation}
\label{sub}
r_{2n+1}r_{2m+1}=r_{2(n+m)+1}\left(\frac{1}{2}\right)^{\frac{1}{K}}=r_{2(n+m)+2} \left(\frac{1}{2}\right)^{\frac{1}{K}-K},
\end{equation}
but we do have
\[r_{2n+1}r_{2m+2} = r_{2(n+m)+2}<r_{2n+1}r_{2m+1}<r_{2(n+m)+1} = r_{2n+1}r_{2m}.\]

We are now in a position to compute elements of the infinitesimal space of $f$. We will be looking for limits of $f(rt_k)/f(t_k)$ where $t_k \downarrow 0$ as $k \to \infty$.

First, let $r\in [r_m,r_{m-1}]$ and $t=r_{2n}$, then $rt\in [r_{2n+m},r_{2n+m-1}]$. So, by \eqref{k}, \eqref{eq:r}, \eqref{eq:C} and \eqref{eq:rs}, we have
\begin{align*}
\frac{f(rt)}{f(t)}&=\frac{C_{2n+m}\left(rr_{2n}\right)^{k_{2n+m}}}{2^{-2n}}\\
&=2^{2n}\left(\frac{1}{2}\right)^{2n+m} r_{2n+m}^{-k_{2n+m}} r_{2n}^{k_{2n+m}}r^{k_{2n+m}}\\
&=\left(\frac{1}{2}\right)^{m} r_{m}^{-k_{2n+m}} r^{k_{2n+m}}\\
&=\begin{cases} 
       \left(\frac{1}{2}\right)^{m} r_{m}^{-\frac{1}{K}} r^{\frac{1}{K}},  &\text{for $m$ even} \\
       \left(\frac{1}{2}\right)^{m} r_{m}^{-K} r^{K},  &\text{for $m$ odd.}
       \end{cases}
\end{align*}

Next, for $t=r_{2n-1}$ and $r\in [r_{2m+2},2^{-K+1/K}r_{2m+1}]$, then $rt\in [r_{2n+m+1},r_{2n+m}]$. So,
again using \eqref{k}, \eqref{eq:r}, \eqref{eq:C} and \eqref{eq:rs}, we have
\begin{align*}
\frac{f(rt)}{f(t)}&=\frac{C_{2n+2m+1}\left(rr_{2n-1}\right)^{k_{2n+2m+1}}}{\left(\frac{1}{2^{2n-1}}\right)}\\
&=2^{2n-1}\left(\frac{1}{2}\right)^{2n+2m+1} r_{2n+2m+1}^{-K} r_{2n-1}^{K}r^{K}\\
&=\left(\frac{1}{2}\right)^{2m+2} r_{2m+2}^{-K} r^{K}.\\
\end{align*}

For $t=r_{2n-1}$ and $r\in [2^{-K+1/K}r_{2m+1},r_{2m}]$, then $rt\in [r_{2n+m},r_{2n+m-1}]$. So, once more using
\eqref{k}, \eqref{eq:r}, \eqref{eq:C} and this time \eqref{sub}, we have
\begin{align*}
\frac{f(rt)}{f(t)}&=\frac{C_{2n+2m}\left(rr_{2n-1}\right)^{k_{2n+2m}}}{\left(\frac{1}{2^{2n-1}}\right)}\\
&=2^{2n-1}\left(\frac{1}{2}\right)^{2n+2m} r_{2n+2m}^{-\frac{1}{K}} r_{2n-1}^{\frac{1}{K}}r^{\frac{1}{K}}\\
&=\left(\frac{1}{2}\right)^{2m+1}\left(\frac{1}{2}\right)^{\left(K-\frac{1}{K}\right)\left(-\frac{1}{K}\right)} r_{2m+1}^{-\frac{1}{K}} r^{\frac{1}{K}}\\
&=\left(\frac{1}{2}\right)^{2m+\frac{1}{K^2}}r_{2m+1}^{-\frac{1}{K}} r^{\frac{1}{K}}\\
&=\left(\frac{1}{2}\right)^{2m}r_{2m}^{-\frac{1}{K}} r^{\frac{1}{K}}
\end{align*}
where at the last step we have used \eqref{rn}.

In each case above, $\frac{f(rt)}{f(t)}$ is independent of $n$. We may therefore consider the two sequences given by $t_k = r_{2k}$ and $t_k = r_{2k-1}$ and the corresponding limit functions:
\begin{align*}
P_1(r)&=\begin{cases} 
\left(\frac{1}{2}\right)^{2m}r_{2m}^{-\frac{1}{K}}r^{\frac{1}{K}}, & r\in[r_{2m},r_{2m-1}]\\[5pt]
\left(\frac{1}{2}\right)^{2m-1}r_{2m-1}^{-K}r^{K}, & r\in[r_{2m-1},r_{2m-2}]
       \end{cases}
       \\
P_2(r)&=\begin{cases} 
\left(\frac{1}{2}\right)^{2m+2}r_{2m+2}^{-K}r^{K}, & r\in[r_{2m+2},2^{-K+1/K}r_{2m+1}]\\[5pt]
\left(\frac{1}{2}\right)^{2m}r_{2m}^{-\frac{1}{K}}r^{\frac{1}{K}}, & r\in[2^{-K+1/K}r_{2m+1},r_{2m}].
       \end{cases}
\end{align*}
These are distinct elements of the infinitesimal space of $f$, and thus $f$ is not simple at $0$.

We then define $F:B^d\to B^d$ in terms of spherical coordinates by
\[ F(r,\sigma) = (f(r), \sigma),\]
for $r\in (0,1)$ and $\sigma \in S^{d-1}$, and by $F(0)=0$. By Lemma \ref{lem:radial} and the fact that $(n-1)$-spheres are removable for quasiregular mappings by \cite[Theorem VII.1.19]{Rickman}, we see that $F$ is quasiconformal with maximal distortion $K^{d-1}$. In particular, any maximal distortion larger than $1$ can be prescribed for $F$. Finally, by the calculation above, we see that the corresponding quasiconformal extensions of $P_1$ and $P_2$ are two distinct elements of the infinitesimal space of $F$.
\end{proof}

\section{Infinitesimal spaces and uqr maps}

In this section we will generalize the result of the previous section to show that we can construct a uniformly quasiconformal map $B^d\to B^d$ which is not simple at the fixed point $0$.

\begin{proof}[Proof of Theorem \ref{thm:main}]
Since every uniformly quasiconformal map which has an attracting fixed point at $0$ can be conjugated by a quasiconformal map to $x\mapsto x/2$, to construct the required uniformly quasiconformal map, we will conjugate $x\mapsto x/2$ by the quasiconformal map $F$ constructed in Proposition \ref{prop:1}. The resulting map will also be radial, and so we first consider the map $[0,1] \to [0,1]$ defined by
\[ h(r) = f^{-1}\left (\frac{ f(r)}{2} \right ),\]
where $f$ is the map from Proposition \ref{prop:1}.

We will give explicit formulas for $h$. To this end, we observe that $h$ acts by mapping the interval $[r_n,r_{n-1}]$ to $[2^{-n},2^{-n+1}]$ under $f$, then mapping to $[2^{-n-1},2^{-n}]$ and finally to $[r_{n+1},r_n]$ by $f^{-1}$.
Therefore, $h$ may be written as 
\begin{equation}
\label{eq:hdef}
  h(r) = \left\{\def\arraystretch{1.2}%
  \begin{array}{ll}
   \left(\frac{C_{2n - 1}}{2C_{2n}}\right)^K r^{K^2}, & r \in [r_{2n - 1}, r_{2n - 2}]\\
  \left(\frac{C_{2n}}{2C_{2n+1}}\right)^{1/K} r^{1/K^2}, & r \in [r_{2n}, r_{2n - 1}]\\
  \end{array}\right.
\end{equation}
for $n \in \N$. 
Using the formulas for $C_n$ from \eqref{eq:C}, this simplifies to
\[
  h(r) = \left\{\def\arraystretch{1.2}%
  \begin{array}{ll}
    2^{(n - 1)K^3 - n/K} r^{K^2}, & r \in [r_{2n - 1}, r_{2n - 2}]\\
  2^{n/K^3 - 1/K - nK} r^{1/K^2}, & r \in [r_{2n}, r_{2n - 1}].\\
  \end{array}\right.
\]
By the construction of $h$, the map defined in spherical coordinates by
\[ H(r,\sigma) = (h(r),\sigma)\]
for $r\in (0,1)$ and $\sigma \in S^{d-1}$ with $H(0)=0$ is uniformly quasiconformal, since it is a quasiconformal conjugate (recall $F$ is quasiconformal) of $x\mapsto x/2$, and has an attracting fixed point at $0$. We need to show that $H$ is not simple at $0$, and it suffices to show this is so for $h$.

We will be looking for limits of $h(rt_k)/h(t_k)$ as $t_k \downarrow 0$, using two sequences corresponding to $t_k = r_{2k}$ and $t_k = r_{2k-1}$, as in the proof of Proposition \ref{prop:1}.

Let $r \in [r_{m+1},r_m]$ and let $t=r_{2n}$.
Then, by \eqref{k}, \eqref{sub} and \eqref{eq:hdef}, we have
\begin{align*}
h\left(rt\right)&=\left(\frac{1}{2}\right)^{\frac{1}{k_{2n+m+2}}} r_{2n+m+1}^{1-\frac{k_{2n+m+1}}{k_{2n+m+2}}} \left(rr_{2n}\right)^{\frac{k_{2n+m-1}}{k_{2n+m+2}}}\\
&=\left(\frac{1}{2}\right)^{\frac{1}{k_{2n+m+2}}} r_{m+1}^{1-\frac{k_{2n+m+1}}{k_{2n+m+2}}} r_{2n}r^{\frac{k_{2n+m-1}}{k_{2n+m+2}}}\\
&=\begin{cases}
\left(\frac{1}{2}\right)^{K} r_{m+1}^{1-K^2} r_{2n}r^{K^2}, &\text{for $m$ even}\\[5pt]
\left(\frac{1}{2}\right)^{\frac{1}{K}} r_{m+1}^{1-\frac{1}{K}^2} r_{2n}r^{\frac{1}{K}^2}, &\text{for $m$ odd}
\end{cases}
\end{align*}
Hence, by \eqref{rn}, for $r \in [r_{m+1},r_m]$ with $m$ even,
\begin{align*}
\frac{h(rr_{2n})}{h(r_{2n})}&=\left(\frac{1}{2}\right)^Kr_{m+1}^{1-K^2}\left ( \frac{r_{2n}}{r_{2n+1}}\right ) r^{K^2}\\
&=\left(\frac{1}{2}\right)^{K-\frac{1}{K}}r_{m+1}^{1-K^2}r^{K^2}\\
&=r_{m}^{1-K^2}r^{K^2}
\end{align*}
and for $r \in [r_{m+1},r_m]$ with $m$ odd,
\begin{align*}
\frac{h(rr_{2n})}{h(r_{2n})}&=\left(\frac{1}{2}\right)^{\frac{1}{K}}r_{m+1}^{1-\frac{1}{K^2}}\left ( \frac{r_{2n}}{r_{2n+1}}\right ) r^{\frac{1}{K^2}}\\
&=r_{m+1}^{1-\frac{1}{K^2}}r^{\frac{1}{K^2}}.
\end{align*}

Next, let $t=r_{2n-1}$. As with $f$ in Proposition \ref{prop:1}, we must further partition the location of $r$.
For $r\in [r_{2m+2},2^{-K+1/K}r_{2m+1}]$ we have $rt\in[r_{2n+2m+1},r_{2n+2m}]$ using \eqref{sub}.
Therefore
\[h(rt)=\left(\frac{1}{2}\right)^K r_{2n+2m+1}^{1-K^2}r_{2n-1}^{K^2}r^{K^2}.\]
Hence,
\begin{align*}
\frac{h(rr_{2n-1})}{h(r_{2n-1})}&=\left(\frac{1}{2}\right)^{K}r_{2m+2}^{1-K^2}r_{2n-1}^{1-K^2}\left ( \frac{r_{2n-1}^{K^2}}{r_{2n}} \right ) r^{K^2}\\
&=\left(\frac{1}{2}\right)^{K}r_{2m+2}^{1-K^2} \left ( \frac{r_{2n-1}}{r_{2n}} \right ) r^{K^2}\\
&=r_{2m+2}^{1-K^2}r^{K^2} 
\end{align*}
by \eqref{rn}.
For $r\in [  2^{-K+1/K}r_{2m+1},r_{2m}]$, we have $rt \in[r_{2n+2m},r_{2n+2m-1}]$ using \eqref{rn}.
Therefore
\[h(rt)=\left(\frac{1}{2}\right)^{\frac{1}{K}} r_{2n+2m}^{1-\frac{1}{K^2}}r_{2n-1}^{\frac{1}{K^2}}r^{\frac{1}{K^2}}.\]
Hence, by \eqref{rn},
\begin{align*}
\frac{h(rr_{2n-1})}{h(r_{2n-1})}&=\left(\frac{1}{2}\right)^{\frac{1}{K}}\left(\frac{1}{2}\right)^{\left(K-\frac{1}{K}\right)\left(1-\frac{1}{K^2}\right)}r_{2m+1}^{1-\frac{1}{K^2}}r_{2n-1}^{1-\frac{1}{K^2}}\left ( \frac{r_{2n-1}^{\frac{1}{K^2}}}{r_{2n}}\right ) r^{\frac{1}{K^2}}\\
&=\left(\frac{1}{2}\right)^{\frac{1}{K}\left(K^2-1+\frac{1}{K^2}\right)}r_{2m+1}^{1-\frac{1}{K^2}}\left ( \frac{r_{2n-1}}{r_{2n}}\right ) r^{\frac{1}{K^2}}\\
&=\left(\left(\frac{1}{2}\right)^{-\frac{1}{K}} r_{2m+1}\right)^{1-\frac{1}{K^2}}r^{\frac{1}{K^2}}\\
&=r_{2m}^{1-\frac{1}{K^2}} r^{\frac{1}{K^2}}.
\end{align*}
Note that all these formulas are independent of $n$. We therefore get two different elements of the infinitesimal space of $h$, given by 
\begin{align*}
Q_1(r)&=
\begin{cases}
r_{2m}^{1-K^2}r^{K^2}, &\text{for $x\in [r_{2m+1},r_{2m}]$}\\[5pt]
r_{2m}^{1-\frac{1}{K^2}}r^{\frac{1}{K^2}},&\text{for $r\in[r_{2m},r_{2m-1}]$}
\end{cases}\\
Q_2(r)&=
\begin{cases}
r_{2m+2}^{1-K^2}r^{K^2},&\text{for $r\in[r_{2m+2},2^{-K+1/K} r_{2m+1}]$}\\[5pt]
r_{2m}^{1-\frac{1}{K^2}} r^{\frac{1}{K^2}}, &\text{for $r\in[2^{-K+1/K} r_{2m+1},r_{2m}].$}
\end{cases}
\end{align*}

We conclude that $h$ is not simple at $0$, and therefore the uniformly quasiconformal mapping $H$ is not simple at $0$.

\end{proof}

We have exhibited two elements of the infinitesimal spaces of $F$ and $H$ at $0$, but there are in fact uncountably many elements.

\begin{proof}[Proof of Corollary \ref{cor:0}]
Recall that $f:[0,1] \to [0,1]$ is the radial part of $F$, and recall the elements $P_1,P_2$ of $T(f,0)$ constructed in Proposition \ref{prop:1}.

First note that there exists $r_0 \in (0,1)$ such that $P_1(r_0)\neq P_2(r_0)$. Since $\frac{f(r_0 t)}{f(t)}$ is a continuous functon of $t>0$, by the Intermediate Value Theorem, it accumulates at every value between $P_1(r_0)$ and $P_2(r_0)$ as $t\rightarrow0$. Given such a value between $P_1(r_0)$ and $P_2(r_0)$, say $\lambda$, there exists a sequence $t_n\rightarrow0$ so that $\frac{f(r_0 t_n)}{f(t_n)}\rightarrow \lambda$.

Since the family $\{ \frac{F(x t_n)}{\rho_F(t_n)} :n\in \N \}$ is normal (compare with \cite[Theorem 2.7]{GMRV}), on a subsequence $\frac{F(x t_n)}{\rho_F(t_n)}$ converges to some quasiconformal mapping $F_\lambda (x)$ with $|F_{\lambda}(x)| = \lambda$ when $|x| = r_0$. As $\lambda$ varies, we obtain uncountably many mappings in the infinitesimal space of $F$ at $0$.

The proof for $H$ is similar.
\end{proof}

We can use Theorem \ref{thm:main} to show that any attracting or repelling fixed point $x_0$ of a uqr map where $i(x_0,f)=1$ can be conjugated to one where the fixed point is not simple.

\begin{proof}[Proof of Corollary \ref{cor:1}]
Assume $x_0$ is an attracting fixed point of the uniformly quasiregular mapping $f$ with $i(x_0,f) = 1$. It was shown in \cite{HM} that in a neighborhood $U$ of $x_0$, there exists a quasiconformal map $L$ such that $L \circ f =  T\circ L$ and $L(x_0)=0$, where $T$ is the map $x \mapsto x/2$. In particular, 
\begin{equation}
\label{eq:cor1a}
L \circ f \circ L^{-1} = T.
\end{equation}
Next, by Theorem \ref{thm:main}, we can conjugate $T$ by the quasiconformal map $F$ of Proposition \ref{prop:1} to obtain the quasiconformal map $H$ of Theorem \ref{thm:main}. In other words, 
\begin{equation}
\label{eq:cor1b}
H = F^{-1} \circ T \circ F.
\end{equation}

Combining \eqref{eq:cor1a} and \eqref{eq:cor1b}, we, locally near $0$, have the equation
\[ H = F^{-1} \circ L \circ f \circ L^{-1} \circ F .\]
We want to extend this to all of $\R^d$.
To that end, let $G: U \to \R^d$ be the quasiconformal mapping defined by $G = F^{-1}\circ L$.  
Choose $R>0$ sufficiently large that $V = B(x_0,R) \setminus \overline{U}$ is a ring domain.
Define a map $g: \R^d \to \R^d$ by
\[g(x) = 
\left \{
\begin{array}{ll}

G(x), & x \in U\\
x, & |x| > R\\
A(x), & x \in \overline{V}\\

\end{array}
\right .
\]
where $A$ is a quasiconformal mapping obtained from Sullivan's Annulus Theorem (see \cite{TV}).  
Then $g\circ f \circ g^{-1}$ is a uniformly quasiregular mapping which agrees with $H$ near $0$ and hence is not simple at $0$.

The case where $x_0$ is repelling follows similarly by observing that a local inverse to $f$ has $x_0$ as an attracting fixed point.

\end{proof}

\end{document}